\documentclass[a4paper,12pt]{article}
\usepackage{amsmath}
\usepackage{amsthm}
\usepackage{epsfig}
\usepackage{psfrag}
\usepackage{subfigure}
\usepackage[mathscr,mathcal]{eucal}
\usepackage{amssymb}
\usepackage{enumerate}

\def \Z {\mathbb Z}

\def \R {\mathbb R}

\def \ind{1\!\!1}

\def\wh{\widehat}

\def\da{\downarrow}
\def\ua{\uparrow}

\def\phi{\varphi}

\def\be{\begin{equation}}
\def\ee{\end{equation}}

\def\bea{\begin{eqnarray}}
\def\eea{\end{eqnarray}}

\def\boP{\mathbf{P}}

\def\cL{\mathcal{L}}

\newtheorem {lemma}{Lemma}

\newtheorem {proposition}{Proposition}

\newtheorem* {theorem*}{Theorem}
\newtheorem* {thm*}{Theorem}
\newtheorem* {lemma*}{Lemma}
\newtheorem* {corollary*}{Corollary}
\newtheorem* {prop*}{Proposition}
\newtheorem* {definition*}{Definition}
\newtheorem* {remark*}{Remark}

\renewcommand{\d}{\,\mathrm d}
\renewcommand{\P}{\mathbb P}
\newcommand{\condP}[2]{\P\left(#1\bigm|#2\right)}
\DeclareMathOperator*{\BES}{BES}

\title{Skorohod-reflection \\ of Brownian Paths and $\BES^3$}

\author{
{\sc B\'alint T\'oth} and {\sc B\'alint Vet\H o}
\\[5pt]
Institute of Mathematics
\\
Technical University of Budapest (BME)
}

\begin{document}

\maketitle

\vskip1cm
\centerline
{\large\it
Dedicated to S\'andor Cs\"org\H o on the occasion of his 60th birthday.}
\vskip1cm

\begin{abstract}
Let $B(t)$, $X(t)$ and $Y(t)$ be independent standard 1d Borwnian motions.
Define $X^+(t)$ and $Y^-(t)$ as the trajectories of the processes $X(t)$ and
$Y(t)$ pushed upwards and, respectively, downwards by $B(t)$, according to
Skorohod-reflection. In the recent paper \cite{warren}, Jon Warren
proves inter alia that $Z(t):= X^+(t)-Y^-(t)$ is a three-dimensional
Bessel-process. In this note, we present an alternative, elementary
proof of this fact.
\end{abstract}

\section{Introduction}
\label{s:introduction}

The study of 1d Brownian trajectories pushed up or down by Skorohod-reflection
on some other Borwnian trajectories (running backwards in time) was initiated
in \cite{soucaliuctothwerner} and motivated in \cite{tothwerner} by the
construction of the object what is today called the Brownian Web, see
\cite{fontesetal}. It turns out that these Brownian paths, reflected on one
another, have very interesting, sometimes surprising properties.
For further studies of Skorohod-reflection of Brownian paths on one
another see also \cite{soucaliucwerner}, \cite{burdzynualart},
\cite{warren} etc. In particular, in \cite{warren} Warren considers
two interlaced families of Brownian paths with paths belonging to the
second family reflected off paths belonging to the first (in Skorohod's
sense) and derives a determinantal formula for the distribution of
coalescing Brownian motions.

A particular case of Warren's formula is the following: fix a Brownian path
and let two other Brownian paths be pushed upwards and respectively
downwards by Skorohod-reflection on the trajectory of the first one.
The difference of the last two will be a three-dimensional Bessel-process.
In the present note, we give an alternative, elementary proof of this fact.

\subsection{Skorohod-reflection}
\label{ss:skorohod}

Let $T\in(0,\infty)$  and $b, x:[0,T)\to \R$ be continuous functions. Assume
$x(0)\ge b(0)$. The construction of the following proposition is due to
Skorohod. Its proof can be found either in \cite{revuzyor} (see Lemma 2.1 in
Chapter VI) or in \cite{soucaliuctothwerner} (see Lemma 2 in  Section 2.1)

\begin{proposition}
\label{prop:skorohod}
\begin{enumerate}[(1)]
\item There exists a unique continuous function $x_{b\ua}:[0,T)\to\R$ with
the following properties
\begin{enumerate}[--]
\item The function $x_{b\ua}-b$ is non-negative.

\item The function $x_{b\ua}-x$ is non-decreasing.

\item The function $x_{b\ua}-x$ increases only when  $x_{b\ua}=b$. That is
\[\int_0^T \ind\{x_{b\ua}(t)\not=b(t)\}\d(x_{b\ua}(t)-x(t))=0.\]
\end{enumerate}

\item The function $t\mapsto x_{b\ua}(t)$ is  given by the construction
\[x_{b\ua}(t)=x(t)+\sup_{0\le s\le t}\big(x(s)-b(s)\big)_-.\]

\item The map $C([0,T))\times C([0,T)) \ni
\big(b(\cdot),x(\cdot)\big)\mapsto\big(b(\cdot),x_{b\ua}(\cdot)\big) \in
C([0,T))\times C([0,T))$ is contunuous in supremum distance.
\end{enumerate}
\end{proposition}

We call the function $t\mapsto x_{b\ua}(t)$ the \emph{upwards
Skorohod-reflection} of $x(\cdot)$ on $b(\cdot)$. As it is remarked in
\cite{soucaliuctothwerner}, the term \emph{Skorohod-pushup} of $x(\cdot)$ by
$b(\cdot)$ would be more adequate. Skorohod-reflection on paths $b(t)=$ const.\
plays a fundamental role in the proper formulation and proof of Tanaka's
formula, see Chapter VI of  \cite{revuzyor}.

The downwards Skorohod-reflection or Skorohod-pushdown is defined for
continuous functions $b, y:[0,T)\mapsto \R$ with $y(0)\le b(0)$ by
\[y_{b\da}:=-\left( (-y)_{(-b)\ua}\right),
\qquad y_{b\da}(t)=y(t)-\sup_{0\le s\le t}\big(y(s)-b(s)\big)_+.\]

Given three continuous trajectories $b,x,y:[0,T)\to\R$ with $y(0)\le b(0)\le
x(0)$, the map $C([0,T))\times C([0,T))\times C([0,T)) \ni \left(b(\cdot),
x(\cdot), y(\cdot)\right)\mapsto \left(b(\cdot), x_{b\ua}(\cdot),
y_{b\da}(\cdot)\right) \in C([0,T))\times C([0,T))\times C([0,T))$ is clearly
continuous in supremum distance.

\subsection{The result}
\label{ss:result}

Let $B(t)$, $X(t)$ and $Y(t)$ be independent standard 1d Brownian motions
starting from $0$ and define

\setlength{\arraycolsep}{.13889em}
\begin{eqnarray}
\label{pushX} X^+(t):=&X_{B\ua}(t), \qquad
\wh X(t)&:=\phantom{-}X^+(t)-B(t),\\[2pt]
\label{pushY} Y^-(t):=&Y_{B\da}(t), \qquad \wh Y(t)&:=-Y^-(t)+B(t).
\end{eqnarray}

We are interested in the difference process
\begin{equation} \label{diff}
Z(t):=X^+(t)-Y^-(t)=\wh X(t)+\wh Y(t).
\end{equation}
It is straightforward that $2^{-1/2}\wh X(t)$ and $2^{-1/2}\wh Y(t)$ are both
standard reflected Brownian motions. They are, of course, strongly dependent.

The following fact is a particular consequence of the main results in
\cite{warren}:

\begin{theorem*}
\label{thm:main} The process $2^{-1/2}Z(t)$ is $\BES^3$, that is standard 3d
Bessel-process.
\[\d Z(t)=2\frac1{Z(t)}\d t +\sqrt2\d W(t), \qquad Z(0)=0.\]
\end{theorem*}

\noindent
In the next section, we present an elementary proof of this fact.
\section{Proof}
\label{s:proof}

\subsection{Discrete Skorohod-reflection}
\label{ss:discrete}

Define the following square lattices inbedded in $\R\times\R$:
\begin{gather}\label{lattices}
\cL:=\{(t,x)\in\Z\times\Z: t+x\text{ is even}\}, \qquad
\cL^*:=\{(t,x)\in\Z\times\Z: t+x\text{ is odd}\}.
\end{gather}
In both of the lattices, the points $(t_1,x_1)$ and $(t_2,x_2)$ are connected
with an edge if and only if $|t_1-t_2|=|x_1-x_2|=1$. Note that $\cL$ and
$\cL^*$ are Whitney-duals of each other.

We define the discrete analogue of the Skorohod-reflection in $\cL$ and
$\cL^*$. Later on, we say that the function $y:[0,T]\cap\Z\to\Z$ is a
\emph{walk} in the lattice $\cL$ or $\cL^*$ if the consecutive elements of the
sequence $(0,y(0)),(1,y(1)),\dots,(T,y(T))$ are edges in $\cL$ or $\cL^*$.

Let $b:[0,T]\cap\Z\to\Z$ and $x:[0,T]\cap\Z\to\Z$ be two walks in the lattices
$\cL$ and $\cL^*$, respectively. Assume that $x(0)\ge b(0)$. An analogue of
Proposition 1 holds in this case, but the proof is even easier.

\begin{proposition}
\begin{enumerate}[(1)]
\item There is a unique walk $x_{b\ua}:[0,T]\cap\Z\to\Z$ in $\cL^*$ with the
following properties:
\begin{enumerate}[--]
\item The function $x_{b\ua}-b$ is non-negative.
\item The function $x_{b\ua}-x$ is non-decreasing.
\item The function $x_{b\ua}-x$ increases only when  $x_{b\ua}=b+1$, i.e.
\[\sum_{t=1}^T \ind\{x_{b\ua}(t)-b(t)>1\}
\left[\big(x_{b\ua}(t)-x(t))-(x_{b\ua}(t-1)-x(t-1)\big)\right]=0.\]
\end{enumerate}

\item The function $t\mapsto x_{b\ua}(t)$ can be expressed as
\[x_{b\ua}(t)=x(t)+\sup_{s\in[0,t]\cap\Z}(x(s)-b(s))_-+1.\]
\end{enumerate}
\end{proposition}

We call the function $t\mapsto x_{b\ua}(t)$ the \emph{discrete upwards
Skorohod-reflection} of $x(\cdot)$ on $b(\cdot)$. The discrete downwards
Skorohod-reflection is defined similarly. If $y:[0,T]\cap\Z\to\Z$ is a walk in
$\cL$ and $b:[0,T]\cap\Z\to\Z$ is a walk in $\cL^*$ with $y(0)\le b(0)$, then
\[y_{b\da}:=-\left((-y)_{(-b)\ua}\right),\qquad
y_{b\da}(t)=y(t)-\sup_{s\in[0,t]\cap\Z}(y(s)-b(s))_+-1.\]

In this paper, we use the same notation for the discrete Skorohod-reflection
and the continuous one (defined as Skorohod-reflection), but it will be always
clear from the context which is the adequate one.

\subsection{Approximation of reflected Brownian motions}

Let $M(t)$ be a random walk on the lattice $\cL$ with jumps from $(t,x)$ to
$(t+1,x+1)$ or $(t+1,x-1)$ with probability $1/2-1/2$ and $M(0)=0$. We define
the random walks $U(t)$ and $L(t)$ on $\cL^*$ with the same transition
probabilities, which are independent of each other and of $M(t)$. The initial
values are $U(0)=1$ and $L(0)=-1$. We extend our walks for non-integral values
of $t$ linearly, so the trajectories are continuous.

Since all these three random walks have steps with mean $0$ and variance $1$,
it follows that
\begin{equation}
\left(\frac{M(nt)}{\sqrt n},\frac{U(nt)}{\sqrt n},\frac{L(nt)}{\sqrt n}\right)
\stackrel{\mathrm d}{\Longrightarrow}(B(t),X(t),Y(t))\qquad(n\to\infty).
\end{equation}

We established earlier that the map
$(b(\cdot),x(\cdot),y(\cdot))\mapsto(b(\cdot),x_{b\ua}(\cdot),y_{b\da}(\cdot))$
is continuous in supremum distance. From Donsker's invariance principle (see
e.g.\ Section 7.6 of \cite{durrett}), we conclude that
\begin{equation}
\left(\frac{M(nt)}{\sqrt n},\frac{U_{M(n\cdot)\ua}(nt)}{\sqrt n},
\frac{L_{M(n\cdot)\da}(nt)}{\sqrt n}\right) \stackrel{\mathrm d}
{\Longrightarrow}(B(t),X^+(t),Y^-(t))
\end{equation}
in distribution as $n\to\infty$. Note that we can use the discrete
Skorohod-reflection to transform $U$ and $L$, because the difference is only
the addition of $1$, which vanishes in the limit. At this point, it suffices to
show that
\[2^{-1/2}\frac{U_{M(n\cdot)\ua}(nt)-L_{M(n\cdot)\da}(nt)}{\sqrt n}\]
converges to a $\BES^3$-process.

For $x,y\in\Z^+$, we define the stochastic matrix
\[
\boP_{xy}
=
\frac yx \cdot
\left\{
\begin{array}{c@{\extracolsep{1em}}l}
\frac12 & \text{if } y=x\\[2pt]
\frac14 & \text{if } |y-x|=1\\[2pt]
0 & \text{otherwise}
\end{array}\right..
\]

It is well known that if $X_n$ is a homogeneous Markov-chain with
transition probabilities ${\left(\boP_{xy}\right)}_{x,y\in\Z^+}$,
then its diffusive limit is $\BES^3$, i.e.\ for every $T>0$ the process
$\sqrt2(n^{-1/2}X_{nt})_{0\le t\le T}$ converges to a 3d Bessel-process
in the Skorohod-topology as $n\to\infty$.  So the proof of our theorem
relies on the following

\begin{lemma}
$\frac12(U_{M\ua}(t)-L_{M\da}(t))$ is a Markov-chain and its
transition matrix is $(\boP_{xy})_{x,y\in\Z^+}$, where $U_{M\ua}$ and
$L_{M\da}$ are discrete Skorohod-reflections.
\end{lemma}

\subsection{Markov-property of the distance of the two reflected walks}

We introduce a different notation for the triple $(M,U_{M\ua},L_{M\da})$, which
is just a linear transformation. Let $K_n:=L_{M\da}(n)$ be the position of the
lower reflected walk. With the definition
$D_n:=\frac12(U_{M\ua}(n)-L_{M\da}(n))$, the distance of the two reflected
walks is $2D_n$. $P_n:=\frac12(M(n)-L_{M\da}(n)-1)$, which means that the
position of $M$ related to the lower walk is $2P_n+1$. The vector
$(K_n,D_n,P_n)$ is clearly a Markov-chain.

We are only interested in the coordinate $D_n$, which turns out to be also
Markov and to have transition matrix $(\boP_{xy})_{xy\in\Z^+}$. To show
this, we have to determine the conditional distribution of $P_n$, because
in certain cases it modifies the transition rules of $D_n$.

\begin{lemma}
The following identities hold
\begin{eqnarray}
&\label{id1}
\P\left(P_n=x\bigm|D_0^n\right)=
\frac1{D_n}\ind(x\in\{0,1,\dots,D_n-1\}),
\\[1ex]
&\label{id2}
\P\left(D_{n+1}=y\bigm|D_0^n\right)=\boP_{D_ny}
\end{eqnarray}
where $D_0^n$ means the sequence of variables $D_0,\dots,D_n$.
\end{lemma}

\begin{proof}
The two identities \eqref{id1}, respectively, \eqref{id2}
of the lemma are proved by a common induction on
$n$. Since $D_0=1$ and $P_0=0$, the case $n=0$ is trivial.

For the induction step, we have to enumerate the possible transitions of the
Markov-chain $(K_n,D_n,P_n)$. For the sake of simplicity, we only prove for
$D_n=D_{n-1}-1$, the other cases are similar. It is easy to check that the
transition $(k,d,p)\to(k+1,d-1,p)$ has probability
$\frac18\ind(p\in\{0,1,\dots,d-2\})$, this will be called type $A$ events. Type
$B$ events are the transitions $(k,d,p)\to(k+1,d-1,p-1)$, which happen with
probability $\frac18\ind(p\in\{1,2,\dots,d-1\})$. No other cases give $d\to
d-1$.

\noindent{\sl Proof of \eqref{id1}}:
Let $x,y\in\Z^+$. We suppose that $y=D_{n-1}-1$.
\begin{eqnarray}
&&
\hskip-15mm
\condP{P_n=x}{D_n=y,D_0^n}=
\\[2pt]
\notag
&=&
\sum_{z\in\Z}\condP{P_n=x}{P_{n-1}=z,D_n=y,D_0^{n-1}}
\condP{P_{n-1}=z}{D_n=y,D_0^{n-1}}
\\[2pt]
\notag
&=&
\sum_{z\in\Z}\frac{\condP{P_n=x,D_n=y}{P_{n-1}=z,D_0^{n-1}}}
{\condP{D_n=y}{P_{n-1}=z,D_0^{n-1}}}\;
\condP{P_{n-1}=z}{D_n=y,D_0^{n-1}}
\\[2pt]
\notag
&=&
\sum_{z=x}^{x+1}\condP{P_n=x,D_n=y}{P_{n-1}=z,D_0^{n-1}}
\frac{\condP{P_{n-1}=z}{D_0^{n-1}}}{\condP{D_n=y}{D_0^{n-1}}}
\\[2pt]
\notag
&=&
\condP{P_n=x,D_n=y}{P_{n-1}=x,D_0^{n-1}}
\frac{\condP{P_{n-1}=x}{D_0^{n-1}}}{\condP{D_n=y}{D_0^{n-1}}}
\\[2pt]
\notag
&&
+\condP{P_n=x,D_n=y}{P_{n-1}=x+1,D_0^{n-1}}
\frac{\condP{P_{n-1}=x+1}{D_0^{n-1}}}{\condP{D_n=y}{D_0^{n-1}}}
\\[2pt]
\notag
&=&
\frac18\ind(x\in\{0,1,\dots,D_{n-1}-2\})
\frac{\frac1{D_{n-1}}\ind(x\in\{0,1,\dots,D_{n-1}-1\})}{\frac14\frac{D_{n-1}-1}{D_{n-1}}}
\\[2pt]
\notag
&&
+\frac18\ind(x\in\{0,1,\dots,D_{n-1}-2\})
\frac{\frac1{D_{n-1}}\ind(x\in\{-1,0,\dots,D_{n-1}-2\})}{\frac14\frac{D_{n-1}-1}{D_{n-1}}}
\\[2pt]
\notag
&=&
\frac1{D_{n-1}-1}\ind(x\in\{0,\dots,D_{n-1}-2\})=\frac1y\ind(x\in\{0,1,\dots,y-1\}).
\end{eqnarray}
First, we used the law of total probability and the definition of conditional
probability and the identity $\P(E|F)/\P(F|E)=\P(E)/\P(F)$ on a conditional
probability space. As remarked at the beginning of this proof, there are only
two cases to reduce the value of $D$, so the sum has only two terms. Then, we
used both inductional hypotheses to evaluate the conditional probabilities. The
remaining steps are obvious.\\

\noindent{\sl Proof of \eqref{id2}}:
We spell out the proof for $D_{n+1}=D_n-1$, the cases $D_{n+1}=D_n$ and
$D_{n+1}=D_n+1$ are similar.
\begin{eqnarray}
&&
\hskip-15mm
\condP{D_{n+1}=D_n-1}{D_0^n}=
\\[2pt]
\notag
&=&
\sum_{x=0}^{D_n-1}\condP{D_{n+1}=D_n-1}{P_n=x,D_0^n}\condP{P_n=x}{D_0^n}
\\[2pt]
\notag
&=&
\sum_{x=0}^{D_n-1}\left(\frac18\ind(x\in\{0,1,\dots,D_n-2\})
+\frac18\ind(x\in\{1,2,\dots,D_n-1\})\right)\frac1{D_n}
\\[2pt]
\notag
&=&
\frac14\frac{D_n-1}{D_n}=\boP_{D_n(D_n-1)}.
\end{eqnarray}
In the second step, only type $A$ and $B$ events can cause the transition
$D_{n+1}=D_n-1$. We applied the first part of this lemma to evaluate the second
conditional probability factor.

\end{proof}

As a consequence, we see that the distribution of $D_{n+1}$ conditioned on
$D_0^n$ depends only on $D_n$, which means that $D_n$ is a Markov-chain with
transition matrix $(\boP_{xy})_{xy}$. From this, the assertion of the theorem
follows.\\

\vfill

\noindent{\bf Acknowledgement:}
The result was originally proved without knowledge of Jon Warren's work.
We thank Wendelin Werner for drawing our attention to the existence of
the paper \cite{warren}. The authors' research is partially supported by the
OTKA (Hungarian National Research Fund) grants K 60708 (for B.T. and
B.V.) and TS 49835 (for B.V.).

\vfill

\hbox{ \phantom{M} \hskip7cm
\vbox{\hsize8cm {\noindent Address of authors:\\
{\sc Institute of Mathematics\\
Technical University of Budapest\\
Egry J\'ozsef u.\ 1\\
H-1111 Budapest, Hungary}\\[5pt]
e-mail:{\tt \{balint,vetob\}{@}math.bme.hu} }}}
\end{document}